\documentclass[11pt]{amsart}

\usepackage[centertags]{amsmath}
\usepackage{amsfonts}
\usepackage{amssymb}
\usepackage{amsthm}
\usepackage[english]{babel}
\usepackage[active]{srcltx}

\usepackage[colorlinks]{hyperref}

\newcommand{\dist}{\operatorname{dist}}

\newcommand{\Csmooth}{$\mathcal{C}^1$-smooth }

\newcommand{\Real}{\mathbb{R}}

\newcommand{\Lip}{\operatorname{Lip}}
\newcommand{\supp}{\operatorname{supp}}

\newtheorem{thm}{Theorem}[section]
\newtheorem{cor}[thm]{Corollary}
\newtheorem{lem}[thm]{Lemma}

\newtheorem{defn}[thm]{Definition}

\numberwithin{equation}{section}

\begin{document}

\title[Smooth extension on Banach spaces]{Smooth extension of functions on a certain class of non-separable Banach spaces}

\author{Mar Jim{\'e}nez-Sevilla and Luis S\'anchez-Gonz\'alez}

\address{Departamento de An{\'a}lisis Matem{\'a}tico\\ Facultad de
Matem{\'a}ticas\\ Universidad Complutense\\ 28040 Madrid, Spain}

\thanks{Supported in part by DGES (Spain) Project MTM2009-07848. L. S\'anchez-Gonz\'alez has  also been supported by grant MEC AP2007-00868}

\email{marjim@mat.ucm.es,
lfsanche@mat.ucm.es}



\date{January 19, 2010}

\maketitle


\begin{abstract}
Let us consider a Banach space $X$ with the property that every real-valued Lipschitz function $f$ can be uniformly approximated by  a Lipschitz, $\mathcal{C}^1$-smooth function $g$ with $\Lip(g)\le C \Lip(f)$ (with $C$ depending only on the space $X$).
This is the case  for a  Banach space $X$  bi-Lipschitz homeomorphic to a subset of $c_0(\Gamma)$, for some set $\Gamma$, such that the coordinate functions of the homeomorphism are $\mathcal{C}^1$-smooth (\cite{Hajek}). Then, we prove
that  for every closed subspace $Y\subset X$ and   every \Csmooth (Lipschitz) function  $f:Y\to\Real$, there is a \Csmooth (Lipschitz, respectively) extension of $f$ to $X$. We also study \Csmooth extensions of  real-valued  functions defined on closed  subsets of $X$. These results extend those  given in \cite{Azafrykeener} to the class of non-separable Banach spaces satisfying the above property.

\end{abstract}


\section{Introduction and main results}

In this note we consider the problem of the extension of a smooth function from a subspace of an infinite-dimensional Banach space to a smooth function on the whole space.  More precisely, if $X$ is an  infinite-dimensional Banach space, $Y$ is a closed subspace of $X$ and $f:Y \to \Real$ is a $\mathcal{C}^k$-smooth function, under what conditions does there exist a $\mathcal{C}^k$-smooth function $F:X\to \Real$ such that $F_{\mid_Y}=f$?
Under the assumption that $Y$ is a complemented subspace of a Banach space, an extension of a smooth function $f:Y\to \Real$ is easily found taking the function $F(x)=f(P(x))$, where $P:X\to Y$ is a continuous linear projection. But this extension does not solve the problem since if a Banach space $X$ is not isomorphic to a Hilbert space, then it
has a closed subspace which is not complemented in $X$ \cite{LindTzaf}.

 C. J. Atkin  in \cite{atkin} extends every smooth function $f$ defined on a finite union of open convex sets in a separable Banach space which does not admit smooth bump functions, provided that for every point in the domain of $f$, the restriction of $f$ to a suitable neighborhood of the point can be extended to the whole space. The most fundamental result  has been given by D. Azagra, R. Fry and L. Keener \cite{Azafrykeener, Azafrykeener2}. They have shown that if $X$ is a  Banach space with separable dual $X^*$, $Y\subset X$ is a closed subspace and $f: Y \to \Real$ is a \Csmooth function, then there exists a $\mathcal{C}^1$-smooth extension $F:X\to \Real$ of $f$.  They proved a similar result when $Y$ is a closed convex subset, $f$ is defined on an open set $U$ containing $Y$ and $f$ is \Csmooth on $Y$ as a function on $X$ (i.e., $f:U \rightarrow \mathbb R$ is differentiable at every point $y\in Y$ and the function $Y\mapsto X^*$ defined as $y \mapsto f'(y)$ is continuous on $Y$).
 For a detailed account of the related theory of (smooth) extensions to $\mathbb R^n$ of smooth functions defined in closed subsets of  $\mathbb R^n$ see \cite{Azafrykeener}. Let us point out that the case of analytic maps is quite different (see \cite{AronBerner}).

The aim of this note is to extend the  results in \cite{Azafrykeener}  to the general setting of Banach spaces where every Lipschitz function can be approximated by a $\mathcal{C}^1$-smooth, Lipschitz function. By using the results of Lipschitz and smooth approximation of Lipschitz mappings given by  P. H\'ajek and M. Johanis \cite{HajJohc0, Hajek} we shall extend the results in \cite{Azafrykeener} to a larger class of  Banach spaces.
We proceed along the same lines as the proof of the separable case \cite{Azafrykeener}. Additionally, we shall use the open coverings given by M. E. Rudin, and the ideas of M. Moulis \cite{Moulis},  P. H\'ajek and M. Johanis \cite{Hajek}.

\smallskip
The notation we use is standard. In addition, we shall follow, whenever possible, the notations given in \cite{Azafrykeener}
and \cite{Hajek}. We denote by $||\cdot||$ the norm considered in $X$ and by $B(x,r)$ the open ball with center $x\in X$ and radius $r>0$ .
If $Y$ is a subspace of $X$ we denote the restriction of a function $f:X \to \Real$ to $Y$ by $f_{\mid_Y}$ and we say that $F:X\to \Real$ is an extension of $f:Y\to \Real$ if $F_{\mid_Y}=f$.
Recall that  $\Lip(h)$ denotes the Lipschitz constant of a Lipschitz function $h:Y\rightarrow \mathbb R$, where $Y$ is a subset of a Banach space $X$. We refer to \cite{Deville} or \cite{fabianhajek} for any other definition.

Before stating the main results, let us define the property $(*)$ as the following Lipschitz and $\mathcal{C}^1$-smooth approximation property for Lipschitz mappings.
\begin{defn}
A Banach space $X$ satisfies  property $(*)$ if there is a constant $C_0$, which only  depends on the space $X$, such that, for any Lipschitz function $f:X\to \Real$ and any $\varepsilon>0$ there is a Lipschitz, \Csmooth function $K:X\to \Real$ such that
\begin{equation*}
|f(x)-K(x)|<\varepsilon \text{ for all } x\in X   \text{ and } \Lip(K)\leq C_0 \Lip(f).
\end{equation*}
\end{defn}
We may equivalently say that $X$ satisfies  property $(*)$ if there is a constant $C_0$, which only depends on $X$, such that for any subset $Y\subset X$, any Lipschitz function $f:Y\to \Real$ and any $\varepsilon>0$ there is a $\mathcal{C}^1$-smooth, Lipschitz function $K:X\to \Real$ such that
\begin{equation*}
|f(y)-K(y)|<\varepsilon \ \text{ for all  } y\in Y  \text{ and } \Lip(K)\le C_0\Lip(f).
\end{equation*}
Indeed, every real-valued Lipschitz function $f$ defined on $Y$ can be extended  to a Lipschitz function on $X$ with the same Lipschitz constant (for instance $F(x)=\inf_{y\in Y}\{f(y)+\Lip(f)||x-y||\}$).

Let us recall that every  Banach space with separable dual satisfies property $(*)$ \cite{Fry1,Hajek} (see also \cite{Azafrykeener}).  J.M. Lasry and P.L. Lions  proved in \cite{LL} that in a Hilbert space $H$ and for
every Lipschitz function $f:H\to\Real$ and every $\varepsilon>0$ there exists a \Csmooth and Lipschitz function
$g:H\to\Real$ such that $|f(x)-g(x)|<\varepsilon$ for every $x\in X$ and $\Lip(g)=\Lip(f)$ (see also \cite{AFLMR}).
 Let us notice that the approximation in \cite{LL} is for bounded functions. However it also works  for unbounded Lipschitz functions. P. H\'ajek and M. Johanis have proved in \cite{HajJohc0} that for any set $\Gamma$,  $c_0(\Gamma)$ satisfies property $(*)$. Also, they give a sufficient condition for $X$ to have property $(*)$: if there is a bi-Lipschitz homeomorphism
$\varphi$ embedding  $X$ into $c_0(\Gamma)$ whose coordinates $e_\gamma^*\circ \varphi$ are $\mathcal{C}^1$-smooth,
then $X$ satisfies property $(*)$ \cite{Hajek}.
More specifically, they showed the following characterization:  a Banach space $X$ has property $(*)$ and
is uniformly homeomorphic to a subset of $c_0(\Gamma)$ (for some set $\Gamma$) if and only if there is a bi-Lipschitz homeomorphism $\varphi$ embedding    $X$ into $c_0(\Gamma)$ whose coordinates $e_\gamma^*\circ \varphi$ are $\mathcal{C}^1$-smooth. Let us note that there is a gap in the proof  that WCG Banach spaces with a $C^1$-smooth and Lipschitz bump function satisfy property $(*)$  given in \cite{FryWCG}  and it is unknown  if the result holds.


 We shall prove that, if $X$ satisfies property $(*)$  and $Y$ is a closed subspace of $X$, then for every   \Csmooth real-valued function $f$ defined on $Y$,  there is a $\mathcal{C}^1$-extension of $f$ to $X$.



%

\begin{thm}\label{extension}
Let $X$ be a  Banach space with property $(*)$. Let $Y\subset X$ be a closed subspace and  $f:Y\to\mathbb{R}$ a \Csmooth function. Then there is a \Csmooth extension of $f$ to $X$.\

Furthermore, if the given \Csmooth function $f$ is Lipschitz on $Y$, then there is  a \Csmooth and Lipschitz extension $H:X\to \mathbb{R}$  of $f$ to $X$ such that $\Lip(H)\leq C \Lip(f)$, where $C$ is a constant depending only on $X$.
\end{thm}

\begin{cor}\label{manifold}
Let $M$ be a paracompact $\mathcal{C}^1$-smooth Banach manifold modeled on  a  Banach space $X$  with property $(*)$ (in particular, any Riemannian manifold), and let $N$ be a closed \Csmooth submanifold of $M$. Then, every \Csmooth function $f:N \to \mathbb{R}$ has a \Csmooth extension to $M$.
\end{cor}

A similar result can be stated,
as in the separable case \cite{Azafrykeener},  if $Y$ is a closed convex subset of $X$, $f$ is defined on an open subset $U$ of $X$ such that  $Y\subset U$ and $f:U \rightarrow \mathbb R$   is \Csmooth on $Y$ as a function on $U$, i.e. $f: U \rightarrow \mathbb R$ is differentiable at every point $y\in Y$ and
the mapping $Y \mapsto X^*$, $y \mapsto f'(y)$ is continuous on $Y$.

%


\begin{thm}\label{extensionset}
Let $X$ be a  Banach space with property $(*)$, $Y\subset X$  a closed convex subset, $U\subset X$ an open set containing $Y$ and  $f:U \to \Real$  a \Csmooth function on $Y$ as a function on $U$. Then, there is a \Csmooth extension of $f_{\mid_Y}$ to $X$.

Furthermore, if the given \Csmooth function $f$ is Lipschitz on $Y$, then there is  a \Csmooth and Lipschitz extension $H:X\to
\mathbb{R}$  of $f_{\mid_Y}$ to $X$ such that $\Lip(H)\leq C \Lip(f_{\mid_Y})$, where $C$ is a constant depending only on $X$.
\end{thm}



%

Finally, we can conclude with the following corollary.
\begin{cor}
Let $X$ be a Banach space such that there is a bi-Lipschitz homeomorphism between $X$ and a subset of $c_0(\Gamma)$, for some set $\Gamma$, whose coordinate functions are $\mathcal{C}^1$-smooth. Let $Y\subset X$ be a closed subspace and  $f:Y\to\mathbb{R}$ a \Csmooth function (respectively, \Csmooth and Lipschitz function). Then there is a \Csmooth extension $H$ of $f$ to $X$ (respectively, a \Csmooth and Lipschitz extension $H$ of $f$ to $X$ with
$\Lip(H)\le C \Lip(f)$, where $C$ is a constant depending only on $X$).
\end{cor}
\begin{cor}
Let $X$ be one of the following Banach spaces:
\begin{itemize}
\item[(i)] a Banach space such that there is a bi-Lipschitz homeomorphism between $X$ and a subset of $c_0(\Gamma)$, for some set $\Gamma$, whose coordinate functions are $\mathcal{C}^1$-smooth,
\item[(ii)] a Hilbert space.
\end{itemize}
Let $Y\subset X$ be a closed convex subset, $U\subset X$ an open set containing $Y$ and  $f:U \to \Real$ be a \Csmooth function on $Y$ as a function on $U$ (respectively, \Csmooth on $Y$ as a function on $U$ and Lipschitz on $Y$). Then, there is a \Csmooth extension $H$ of $f_{\mid_Y}$ to $X$ (respectively, \Csmooth and Lipschitz extension $H$ of $f_{\mid_Y}$ to $X$ with
$\Lip(H)\le C \Lip(f_{\mid_Y})$, where $C$ is a constant depending only on $X$).
\end{cor}

In Appendix A, we study under what conditions a   real-valued function defined on a closed, non-convex subset $Y$ of a Banach space $X$ with property $(*)$ can be extended to a  \Csmooth function on $X$.

\section{The proofs }
The first result we shall need is the existence of \Csmooth and Lipschitz partitions of unity on Banach spaces satisfying property $(*)$.
Recall that a Banach space $X$ admits  \Csmooth and Lipschitz partitions of unity when for every open cover
$\mathcal{U}=\{U_r\}_{r\in \Omega}$
 of $X$ there is a collection of $\mathcal{C}^1$-smooth, Lipschitz  functions $\{\psi_i\}_{i\in I}$ such that
 (1)   $\psi_i\ge 0$ on $X$ for every $i\in I$, (2) the family $\{\supp (\psi_i)\}_{i\in I}$ is locally finite, where $\supp (\psi_i) =
 \overline{\{x\in X: \psi_i(x)\neq 0\}}$, (3) $\{\psi_i\}_{i\in I}$ is subordinated to $\mathcal{U}=\{U_r\}_{r\in \Omega}$, i.e. for
 each $i\in I$ there is $r\in \Omega$ such that $\supp (\psi_i)\subset  U_{r}$ and (4) $\sum_{i\in I} \psi_i(x)=1$ for every
 $x\in X$. Also let us denote by $\dist(A,B)$ the distance between two sets $A$ and $B$, that is to say
 $\inf \{||a-b||: a\in A, \ b\in B\}$.


The following lemma gives us the tool to generalize the construction of suitable open coverings on
a Banach space, which will be the key to obtain a generalization of the smooth extension result  given in \cite{Azafrykeener}.
\begin{lem}(See M.E. Rudin \cite{Rudin}) \label{Rudin}
Let $E$ be a metric space, $\mathcal{U}=\{U_r\}_{r\in\Omega}$ be an open covering of $E$. Then, there are open refinements $\{V_{n,r}\}_{n\in\mathbb{N},r\in\Omega}$ and $\{W_{n,r}\}_{n\in\mathbb{N},r\in\Omega}$ of $\mathcal{U}$ satisfying  the following properties:
\begin{itemize}
\item[(i)] $V_{n,r}\subset W_{n,r}\subset U_r$ for all $n\in\mathbb{N}$ and $r\in \Omega$,
\item[(ii)] $\dist(V_{n,r},E\setminus W_{n,r})\geq  1/2^{n+1}$ for all $n\in\mathbb{N}$ and $r\in\Omega$,
\item[(iii)] $\dist(W_{n,r},W_{n,r'})\geq 1/2^{n+1}$ for any $n\in\mathbb{N}$ and $r,r'\in\Omega$, $r\not=r'$,
\item[(iv)] for every $x\in E$ there is an open ball $B(x,s_x)$ of $E$ and a natural number $n_x$ such that
     \begin{enumerate}
         \item[(a)] if $i>n_x$, then $B(x,s_x)\cap W_{i,r}=\emptyset$ for any $r\in\Omega$,
         \item[(b)] if $i\leq n_x$, then $B(x,s_x)\cap W_{i,r}\neq\emptyset$ for at most one $r\in\Omega$.
     \end{enumerate}
\end{itemize}
\end{lem}

P. H\'ajek and M. Johanis \cite{Hajek} proved that if a Banach space $X$ satisfies property (*) then $X$ admits
\Csmooth and Lipschitz partitions of unity,  which is, in turn, equivalent to the existence of a $\sigma$-discrete basis
$\mathcal{B}$ of the topology of $X$ such that for every $B\in \mathcal{B}$ there is a $\mathcal{C}^1$-smooth and Lipschitz  function $\psi_B:X\rightarrow \mathbb [0,1]$
with $B=\psi^{-1}(0,\infty)$ (\cite{Hajek}, see also \cite{JTorZiz}). It is worth noting that given an open covering $\{U_r\}_{r\in \Omega}$ of $X$,
it is not always possible to obtain a \Csmooth and Lipschitz partition of unity $\{\psi_r\}_{r\in \Omega}$ with the same
set of indexes such that
$\supp(\psi_r)\subset U_r$. For example, if $A$ is a non-empty, closed subset of $X$, $W$ is an open subset of $X$ such that
$A\subset W$ with $\dist(A,X\setminus W)=0$, and  $\{\psi_1, \psi_2\}$ is a \Csmooth  partition of
unity subordinated to $\{W, X\setminus A\}$, then $\psi_1(A)=1$ and $\psi_1(X\setminus W)=0$ and thus $\psi_1$ is not Lipschitz.
 Nevertheless, in order to prove Theorem \ref{laclave} we only need the following statement.
\begin{lem}\label{partition:unity}
Let $X$ be a Banach space with  property $(*)$.
Then, for every $\{U_r\}_{r\in\Omega}$ open covering of $X$, there is an  open refinement $\{W_{n,r}\}_{n\in\mathbb{N},r\in\Omega}$ of $\{U_r\}_{r\in\Omega}$ satisfying the properties  of  Lemma \ref{Rudin}, and there is a Lipschitz and \Csmooth partition of unity $\{\psi_{n,r}\}_{n\in\mathbb{N}, r\in\Omega}$ such that $\supp (\psi_{n,r})\subset W_{n,r}\subset U_r$  and $\Lip( \psi_{n,r})\le C_0 2^5(2^n-1)$ for every $n\in\mathbb{N}$ and $r\in\Omega$.
\end{lem}
\begin{proof}



Let us consider  an open covering $\{U_r\}_{r\in\Omega}$ of $X$. By Lemma \ref{Rudin}, there are open refinements
$\{V_{n,r}\}_{n\in\mathbb{N},r\in\Omega}$ and $\{W_{n,r}\}_{n\in\mathbb{N},r\in\Omega}$ of $\{U_r\}_{r\in\Omega}$ satisfying the properties (i)-(iv) of  Lemma \ref{Rudin}. Consider the distance function $D_n(x)=\dist(x, X \setminus \bigcup_{r\in\Omega} W_{n,r})$ which is $1$-Lipschitz. By applying property $(*)$, there is a  $\mathcal{C}^1$-smooth, $C_0$-Lipschitz function
$g_n:X \to \mathbb R$ such that $|g_n(x)-D_n(x)|<\frac{1}{2^{n+3}}$ for every $x\in X$. Thus $g_n(x)>\frac{1}{2^{n+2}}$ whenever
$x\in  \bigcup_{r\in\Omega}V_{n,r}$ and $g_n(x)<\frac{1}{2^{n+3}}$ whenever $x\in X \setminus \bigcup_{r\in\Omega} W_{n,r}$.
By composing $g_n$ with a suitable $\mathcal{C}^\infty$-smooth function $\varphi_n:\mathbb R \rightarrow [0,1]$ with $\Lip(\varphi_n)
\le 2^{n+4}$ we obtain a \Csmooth function $h_n:=\varphi_n(g_n)$ that is zero on an open set including $X \setminus \bigcup_{r\in\Omega} W_{n,r}$,
 ${h_n}_{\mid_{ \bigcup_{r\in\Omega} V_{n,r}}}\equiv 1$ and $\Lip(h_n)\le C_0 2^{n+4}$.
Now, let us define
\begin{equation*}H_1=h_1,  \text{ and } H_n=h_n(1-h_1)\cdots(1-h_{n-1}) \text{ for } n\ge 2.
\end{equation*}
It is clear that $\sum_n H_n(x)=1$ for all $x\in X$. Since $\supp(h_n)\subset \bigcup_{r\in\Omega} W_{n,r}$ and
$\overline{W}_{n,r}\cap \overline{W}_{n,r'}=\emptyset$ for every $n\in \mathbb{N}$ and $r\neq r'$, we can write $h_n=\sum_{r\in\Omega} h_{n,r}$, where $h_{n,r}(x)=h_n(x)$ on $W_{n,r}$ and $\supp (h_{n,r})\subset W_{n,r}$. Notice that $\Lip(h_{n,r})\le \Lip(h_n)\le C_0 2^{n+4}$. Let us define, for every $r\in \Omega$,
\begin{equation*}
\psi_{1,r}=h_{1,r}, \text{ and }  \psi_{n,r}=h_{n,r}(1-h_1)\cdots(1-h_{n-1}) \text{ for each $n\ge 2$.}
 \end{equation*}
 The functions $\{\psi_{n,r}\}_{n\in\mathbb{N},r\in \Omega}$ satisfy that
\begin{itemize}
\item[(i)] they are \Csmooth and Lipschitz, with $\Lip(\psi_{n,r})\le C_0\sum_{i=5}^{n+4}2^{i}= C_02^5(2^n-1)$,
\item[(ii)]  $\supp (\psi_{n,r})\subset \supp (h_{n,r})\subset W_{n,r}$,  and
\item[(iii)]  for every $x\in X$,
\begin{equation*}\sum_{n\in\mathbb{N},r\in\Omega} \psi_{n,r}(x)=\sum_{r\in\Omega} \psi_{1,r}(x)+
\sum_{n\ge 2}\left(\sum_{r\in\Omega}h_{n,r}(x)\right)\prod_{i=1}^{n-1}(1-h_i(x))=\sum_{n\in\mathbb{N}}H_n(x)=1.
\end{equation*}
\end{itemize}
\end{proof}
An alternative proof of  Theorem \ref{laclave} uses the existence of the $\sigma$-discrete basis for  $X$ previously mentioned and the construction, as above, of suitable $\mathcal{C}^1$-smooth and Lipschitz partitions of unity subordinated to any subfamily of this basis.

The following lemma is a necessary modification of property $(*)$ to show the main results.

\begin{lem}\label{lemma:approximation}
Let $X$ be a Banach space with the property $(*)$. Then for every {subset} $Y\subset X$, every continuous function $F:X\to\Real$ such that $F_{\mid_Y}$ is Lipschitz, and every $\varepsilon>0$, there exists a \Csmooth function $G:X\to\Real$ such that
\begin{enumerate}
\item[(i)] $|F(x)-G(x)|\le\varepsilon$ for all $x \in X$, and
\item[(ii)] $\Lip(G_{\mid_Y})\le C_0\Lip(F_{\mid_Y})$. Moreover, $||G'(y)||_{X^*}\le C_0 \Lip(F_{\mid_Y})$ for all $y\in Y$,
where $C_0$ is the constant given by property $(*)$.
\item[(iii)] In addition, if $F$ is Lipschitz on $X$, there exists a constant $C_1\ge C_0$ that depends only on $X$, such that the function $G$ can be chosen to be Lipschitz on $X$ and $\Lip(G)\leq C_1 \Lip(F)$.
\end{enumerate}
\end{lem}

\begin{proof}
Assume that the function $F:X\to\Real$ is continuous on $X$ and $F_{\mid_Y}$ is Lipschitz. Since $X$ admits \Csmooth partitions of unity,   there is (by  \cite[Theorem VIII 3.2]{Deville}) a \Csmooth function $h:X\to\Real$, such that $|F(x)-h(x)|<\varepsilon$ for all $x\in X$. Let $\widetilde{F}:X\to\Real$ be a Lipschitz extension of $F_{\mid_Y}$ to $X$ with $\Lip(\widetilde{F})=\Lip(F_{\mid_Y})$. Let us apply  property $(*)$ to $\widetilde{F}$ to obtain a $\mathcal{C}^1$-smooth, Lipschitz function $g:X\to \Real$ such that $|\widetilde{F}(x)-g(x)|<{\varepsilon}/{4}$ for every $x\in X$, and $\Lip(g)\le C_0 \Lip(F_{\mid_Y})$.\
Consider the open sets $A=\{x\in X: \, |F(x)-\widetilde{F}(x)|<{\varepsilon}/{4}\}$, $B=\{x\in X: \,  |F(x)-\widetilde{F}(x)|<\varepsilon/2\}$ in $X$ and the closed set
$C=\{x\in X: \, |F(x)-\widetilde{F}(x)|\le {\varepsilon}/{4}\}$ in $X$. Then $Y\subset A\subset C \subset B$.
 By \cite[Proposition VIII 3.7]{Deville} there is a \Csmooth function $u:X\to [0,1]$ such that
\begin{equation*}
u(x)=
\begin{cases}
1 & \text{if $x\in C$},\\
0 & \text{if $x\in X\setminus B.$}
\end{cases}
\end{equation*}

Let us define $G:X\to \Real$ as
\begin{equation*}
G(x):=u(x)g(x)+ (1-u(x))h(x).
\end{equation*}
It is clear that $G$ is a \Csmooth function. Since $u(x)=0$ for all $x\in X\setminus B$, we deduce that
  $|F(x)-G(x)|=|F(x)-h(x)|<\varepsilon$ for all $x\in X\setminus B$.
  Now, if $x\in B$, then $|F(x)-G(x)|\leq u(x)|F(x)-g(x)|+(1-u(x))|F(x)-h(x)|\le u(x)(|F(x)-\widetilde{F}(x)|+|\widetilde{F}(x)-g(x)|)+(1-u(x))|F(x)-h(x)|\le\varepsilon$. Finally, since $u(x)=1$ and $G(x)=g(x)$ for every $x\in A$, we obtain that
  $\Lip(G_{\mid_Y})=\Lip(g_{\mid_Y})\le  C_0 \Lip(F_{\mid_Y})$ and
   $||G'(y)||_{X^*}=||g'(y)||_{X^*}\leq C_0 \Lip(F_{\mid_Y})$ for all $y\in Y$.

Let us now assume that $F$ is Lipschitz on $X$. Let us apply property $(*)$ to $F$ and $\widetilde{F}$ (where $\widetilde{F}$ is a Lipschitz extension of  $F_{\mid_Y}$ to $X$ with $\Lip(\widetilde{F})=\Lip(F_{\mid_Y})$). Thus, we obtain \Csmooth and Lipschitz functions $g,\ h:X\to \Real$ such that
\begin{enumerate}
\item[(a)] $|\widetilde{F}(x)-g(x)|<{\varepsilon}/{4}$, for all $x\in X$,
\item[(b)] $|F(x)-h(x)|<\varepsilon$, for every $x\in X$, and
\item[(c)] $\Lip(g)\le C_0 \Lip(F_{\mid_Y})$ and $\Lip(h)\le C_0 \Lip(F)$.
\end{enumerate}
We  take again the subsets $A,\ B$ and  $C$ as in the previous case. Notice that $\dist(C,X\setminus B)\ge \frac{\varepsilon}{4(\Lip(F)+\Lip(F_{\mid_Y}))}=\varepsilon'$. Let us prove that there is a $\mathcal{C}^1$-smooth, Lipschitz function $u:X\to[0,1]$ such that $u(x)=1$ on $C$ and $u(x)=0$ on $X\setminus B$, with $\Lip(u)\le \frac{ 9 C_0(\Lip(F)+\Lip(F_{\mid_Y}))}{\varepsilon}$. Let us consider $0<r\le{\varepsilon'}/4$, and the distance function $D:X\to \Real$,  $D(x)=\dist(x,C)$. Since the function $D$ is $1$-Lipschitz, we apply property $(*)$ to obtain a $\mathcal{C}^1$-smooth, Lipschitz function $R:X\to\Real$ such that  $\Lip(R)\le C_0$ and $|D(x)-R(x)|<r$ for all $x\in X$.  Also, let us take a \Csmooth and Lipschitz function $\varphi:\Real\to [0,1]$ with (i) $\varphi(t)=1$ whenever $|t|\le {r}$,  (ii) $\varphi(t)=0$ whenever  $|t| \ge \varepsilon' - {r}$ and (iii)  $\Lip(\varphi)\le  \frac{9}{8(\varepsilon'-2r)} \le \frac{9}{4\varepsilon'}$.  Next, we define the \Csmooth function $u:X\to [0,1]$, $u(x)=\varphi(R(x))$. Notice that $\Lip(u)\leq \frac{9 C_0 (\Lip(F)+\Lip(F_{\mid_Y}))}{\varepsilon}$.

Let us now consider $G:X\to \Real$ as
\begin{equation*}
G(x)=u(x)g(x)+(1-u(x))h(x).
\end{equation*}
Clearly $G$ is $\mathcal{C}^1$-smooth on $X$. We follow the above proof to obtain that
$|F(x)-G(x)|<\varepsilon$ on $X$, $\Lip(G_{\mid_Y})=\Lip(g_{\mid_Y})\le  C_0 \Lip(F_{\mid_Y})$ and $||G'(y)||_{X^*}\le C_0 \Lip(F_{\mid_Y})$ for all $y\in Y$. Additionally, if $x\in X\setminus \overline{B}$, then $u(x)=0$, $G(x)=h(x)$, and $||G'(x)||_{X^*}=||h'(x)||_{X^*}\le C_0\Lip(F)$. For $x\in \overline{B}$, we have
   \begin{align*}
 & ||G'(x)||_{X^*}\le&\\
 & \le||g(x)u'(x)+h(x)(1-u)'(x)||_{X^*} +  ||u(x)g'(x)+(1-u(x))h'(x)||_{X^*} \le&\\
&\le ||(g(x)-F(x))u'(x)+ (h(x)-F(x))(1-u)'(x)||_{X^*}+C_0\Lip(F)\le&\\
&\le (|g(x)-\widetilde{F}(x)|+|\widetilde{F}(x)-F(x)|+|h(x)-F(x)|)||u'(x)||_{X^*}+C_0\Lip(F)\le&\\
&  \le \frac{7\varepsilon}{4}\cdot \frac{9 C_0(\Lip(F)+\Lip(F_{\mid_Y}))}{\varepsilon}+C_0\Lip(F)\le
33 C_0 \Lip(F).&
\end{align*}
We define $C_1:=33 C_0$ and obtain that $\Lip(G)\le C_1\Lip(F)$.


\end{proof}


The following approximation result is the key to prove Theorem \ref{extension}. Recall that the separable case was given in \cite[Theorem 1]{Azafrykeener}.

\begin{thm}\label{laclave} Let $X$ be a  Banach space with property $(*)$, and $Y\subset X$ a closed subspace. Let $f:Y\to \mathbb{R}$ be a \Csmooth function, and $F$ a continuous extension of $f$ to $X$.  Then, for every $\varepsilon>0$ there exists a \Csmooth function $G:X\longrightarrow\mathbb{R}$ such that if $g=G_{\mid_Y}$ then
\begin{enumerate}
\item[{(i)}] $|F(x)-G(x)|<\varepsilon$ on $X$, and
\item[{(ii)}] $\|f'(y)-g'(y)\|_{Y^*}<\varepsilon$ on $Y$.
\item[{(iii)}] Furthermore, if $f$ is Lipschitz on $Y$ and $F$ is a Lipschitz extension of $f$ to $X$, then the function $G$ can be chosen to be Lipschitz on $X$ and $\Lip(G)\leq C_2 \Lip(F)$, where $C_2$ is a constant only depending on $X$.
\end{enumerate}
\end{thm}

\begin{proof} We follow the steps given in the proof of the separable case with the necessary modifications.
Notice that by the Tietze theorem, a continuous extension $F$ of $f$ always exists.
Since $X$ is a Banach space, $Y\subset X$ is a closed subspace and $f'$ is a continuous function on $Y$, there exists $\{B(y_\gamma,r_\gamma)\}_{\gamma\in\Gamma}$  a covering of $Y$ by open balls of $X$,    with centers $y_\gamma \in Y$,  $0\not\in \Gamma$,   such that
\begin{equation}  \label{condi}
\|f'(y_\gamma)-f'(y)\|_{Y^*}<\frac{\varepsilon}{8C_0}
\end{equation}
on   $B(y_{\gamma},r_\gamma)\cap Y$, where $C_0$ is the positive constant given by
   property $(*)$ (which depends only on $X$). We denote $B_{\gamma}:=B(y_\gamma,r_\gamma)$.

Let us define $T_\gamma$ an  extension of the first order Taylor Polynomial
   of $f$ at $y_\gamma$ given by
$T_\gamma(x)=f(y_\gamma)+H(f'(y_\gamma))(x-y_\gamma)$, for $x\in X$, where $H(f'(y_\gamma))\in X^*$ denotes a
Hahn-Banach extension of $f'(y_\gamma)$ with the same norm, i.e. $||H(f'(y_\gamma))||_{X^*}=||f'(y_\gamma)||_{Y^*}$.
Notice that  $T_\gamma$ satisfies the following properties:
    \begin{enumerate}
    \item[(B.1)] $T_\gamma$ is $\mathcal{C}^\infty$-smooth on $X$,
    \item[(B.2)] $T'_\gamma(x)=H(f'(y_\gamma))$ for all $x\in X$, \,   $T'_\gamma(y)\mid_Y=f'(y_\gamma)$ for every $y\in Y$, and
    \item[(B.3)] from (B.2), \eqref{condi} and the fact that  $B_{\gamma}\cap Y$ is convex, we deduce  $\Lip((T_\gamma-F){\mid_{B_{\gamma}\cap Y}})\le\frac{\varepsilon}{8C_0}.$
    \end{enumerate}

 Since $F:X\to\Real$ is a continuous function, and $X$ admits \Csmooth partitions of unity, there is a \Csmooth function $F_0:X\to\Real$ such that $|F(x)-F_0(x)|<\frac{\varepsilon}{2}$ for every $x\in X$.

Let us denote $B_0:=X\setminus Y$, $\Sigma:=\Gamma\cup \{0\}$, and  $\mathcal{C}:=\{B_{\beta} : \beta\in \Sigma\}$, which is a covering of $X$.  By  Lemma \ref{partition:unity}, there is an open refinement  $\{W_{n,\beta}\}_{n\in\mathbb{N},\beta\in\Sigma}$ of $\mathcal{C}=\{B_\beta : \beta\in\Sigma\}$ satisfying the four properties  of  Lemma \ref{Rudin}. In particular, $W_{n,\beta}\subset B_{\beta}$ and for each $x\in X$ there is an open ball $B(x,s_x)$ of $X$ with center $x$ and radius $s_x>0$, and a natural number $n_x$ such that
     \begin{enumerate}
         \item if $i>n_x$, then $B(x,s_x)\cap W_{i,\beta}=\emptyset$ for every $\beta\in\Sigma$,
         \item if $i\leq n_x$, then $B(x,s_x)\cap W_{i,\beta}\neq\emptyset$ for at most one $\beta\in \Sigma$.
     \end{enumerate}
There is, by Lemma \ref{partition:unity}, a $\mathcal{C}^1$-smooth and Lipschitz partition of unity $\{\psi_{n,\beta}\}_{n\in\mathbb{N},\beta\in\Sigma}$ such that $\supp (\psi_{n,\beta}) \subset
W_{n,\beta}\subset B_\beta$ and $\Lip(\psi_{n,\beta})\le C_02^5(2^n-1)$ for every $(n,\beta)\in \mathbb N \times \Sigma$.


Let us define $L_{n,\beta}:=\max\{\Lip(\psi_{n,\beta}),1\}$ for every $n\in\mathbb{N}$ and $\beta\in\Sigma$.
Now, for every $n\in\mathbb{N}$ and  $\gamma\in\Gamma$, we apply Lemma \ref{lemma:approximation} to $T_\gamma-F$ on $B_{\gamma}\cap Y$ to obtain a \Csmooth   map $\delta_{n,\gamma}:X\longrightarrow \mathbb{R}$ so that
\begin{equation} \tag{C.1}
|T_\gamma(x)-F(x)-\delta_{n,\gamma}(x)|<\frac{\varepsilon}{2^{n+2}L_{n,\gamma}} \ \text{ for every } x\in X
\end{equation}
and
\begin{equation} \tag{C.2}
 \|\delta'_{n,\gamma}(y)\|_{X^*}\le \frac{\varepsilon}{8} \ \text{ for every } y\in B_\gamma\cap Y.
\end{equation}
From inequality  \eqref{condi}, (B.2) and (C.2),  we have  for all $y\in B_{\gamma}\cap Y$,
\begin{equation*}
\|T'_\gamma(y)-f'(y)-\delta'_{n,\gamma}(y)\|_{Y^*} \leq \|T_\gamma'(y)-f'(y)\|_{Y^*}+\|\delta_{n,\gamma}'(y)\|_{Y^*} < \frac{\varepsilon}{4}.
\end{equation*}
Notice that in the above inequality we consider the norm on $Y^*$ (i.e. the norm of the  functional restricted to $Y$).
Let us define
\begin{equation}\label{deltanbeta}
\Delta_\beta^n(x)=
\begin{cases}
F_0(x) & \text{ if } \beta=0, \\
T_\beta(x)-\delta_{n,\beta}(x)  & \text{ if } \beta\in \Gamma.
\end{cases}
\end{equation}
Thus,
$|\Delta_\beta^n(x)-F(x)|<\frac{\varepsilon}{2}$ whenever $n\in \mathbb{N}$,  $\beta\in\Sigma$ and  $x\in B_\beta$.
We now define
\begin{equation*}
G(x)=\sum_{(n,\beta)\in\mathbb{N} \times \Sigma} \psi_{n,\beta}(x)\Delta^n_\beta(x).
\end{equation*}
Since $\{\psi_{n,\beta}\}_{n\in\mathbb{N},\beta\in\Sigma}$ is locally finitely nonzero, then $G$ is $\mathcal{C}^1$-smooth.
Now, if $x\in X$ and $\psi_{n,\beta}(x)\not=0$, then $x\in B_\beta$ and thus
\begin{equation*}
|G(x)-F(x)|\leq \sum_{(n,\beta)\in\mathbb{N} \times \Sigma}\psi_{n,\beta}(x)|\Delta_\beta^n(x)-F(x)|\leq \sum_{(n,\beta)\in\mathbb{N} \times \Sigma} \psi_{n,\beta}(x)\frac{\varepsilon}{2}<\varepsilon.
\end{equation*}

Let us now estimate the distance between the derivatives. From the definitions given above, notice that:
\begin{enumerate}
\item[(D.1)] Since $\sum_{\mathbb N\times \Sigma}\psi_{n,\beta}(x)= 1$  for all $x\in X$, we have that  $\sum_{\mathbb N\times \Sigma}\psi'_{n,\beta}(x)=0$ for all $x\in X$.
\item[(D.2)] Thus, we can write $f'(y)=\sum_{\mathbb N\times \Sigma} (\psi'_{n,\beta}(y))_{\mid_Y} f(y)+\sum_{\mathbb N\times \Sigma}\psi_{n,\beta}(y)f'(y)$, for every $y\in Y$.
\item[(D.3)]   $\supp (\psi_{n,0})\subset B_{0}= X\setminus Y$, for all $n$.
\item[(D.4)] $G'(x)=\sum_{\mathbb N\times \Sigma}\psi'_{n,\beta}(x)\Delta_\beta^n(x)+\sum_{\mathbb N\times \Sigma}\psi_{n,\beta}(x)(\Delta_{\beta}^n)'(x)$, for all $x\in X$.
\item[(D.5)] If $g=G_{\mid_Y}$ then $g'(y)=\sum_{\mathbb N\times \Sigma}\psi'_{n,\beta}(y)_{\mid_Y}\Delta_\beta^n(y)+\sum_{\mathbb N\times \Sigma}\psi_{n,\beta}(y)(\Delta_{\beta}^n)'(y)_{\mid_Y} $, for every $y\in Y$.
\item[(D.6)] Properties (1) and (2) of the open refinement $\{W_{n,\beta}\}$ imply that for every $x\in X$  and $n\in\mathbb{N}$,
there is at most one $\beta\in\Sigma$, which we shall denote by $\beta_x(n)$,  such that $x\in \supp  (\psi_{n,\beta})$.
In the case that $y\in Y$, then $\beta_y(n)\in\Gamma$.  We define $F_x:=\{(n,\beta)\in \mathbb N\times \Sigma: x\in \supp  (\psi_{n,\beta})\}$. In particular, $F_y\subset \mathbb N\times \Gamma$, whenever $y\in Y$.
\end{enumerate}
We obtain, for $y\in Y,$
\begin{align} 
 &\|g'(y)-f'(y)\|_{Y^*}\leq &   \label{f'-g'} \\  \notag
 &\leq \sum_{(n,\beta)\in F_y}\|\psi'_{n,\beta}(y)\|_{Y^*} |T_\beta(y)-f(y)-\delta_{n,\beta}(y)| &\\ \notag
&  \qquad \qquad \qquad \qquad \qquad  \qquad +\sum_{(n,\beta)\in F_y}\psi_{n,\beta}(y) \|T'_\beta(y)-f'(y)-\delta'_{n,\beta}(y)\|_{Y^*}\le &\\ \notag
& \le \sum_{\{n: (n,\beta_y(n))\in F_y\}}L_{n,\beta_y(n)}| T_{\beta_y(n)}(y)-f(y)-\delta_{n,\beta_y(n)}(y)|+ &\\  \notag
 & \qquad \qquad  \quad   +\sum_{\{n: (n,\beta_y(n))\in F_y\}}\psi_{n,\beta_y(n)}(y)\|T'_{\beta_y(n)}(y)-f'(y)-\delta'_{n,\beta_y(n)}(y)\|_{Y^*}\leq &\\  \notag
&\le \sum_{\{n: (n,\beta_y(n))\in F_y\}}(L_{n,\beta_y(n)}\,\frac{\varepsilon}{2^{n+2}L_{n,\beta_y(n)}}+
\psi_{n,\beta_y(n)}(y)\frac{\varepsilon}{4})\leq \frac{\varepsilon}{4}+\frac{\varepsilon}{4}<\varepsilon,&
\end{align} where  all the functionals involved are considered restricted to $Y$.
\smallskip

Let us now consider the case when $f$ is \Csmooth and Lipschitz on $Y$ and $F$ is a Lipschitz extension of $f$ on $X$.
 In this case, we can assume that $f$ is not constant (otherwise the assertion is trivial) and thus $\Lip(F)\ge\Lip (f)>0$.
Let us fix $0<\varepsilon < \Lip(F)$.
If we follow the above construction  for the open covering $\{B_{\beta}\}_{\beta\in\Sigma}$ of $X$
satisfying the conditions (\ref{condi}), (B.1), (B.2) and (B.3), we additionally obtain
\begin{enumerate}
\item[(B.4)] $T_\beta-F$ is Lipschitz on $X$ and  $\Lip(T_\beta-F)\le \Lip(f)+\Lip(F)$ for every $\beta \in \Gamma$.
\end{enumerate}
Also, the construction of the open refinement
$\{W_{n,\beta}\}_{n\in \mathbb N, \beta\in \Sigma}$ of $\mathcal{C}=\{B_{\beta}\}_{\beta\in\Sigma}$,
the Lipschitz partition of unity $\{\psi_{n,\beta}\}_{n\in\mathbb{N},\beta\in\Sigma}$ and the definition of $L_{n,\beta}$
are similar to the previous case (recall that $\Sigma=\Gamma\cup\{0\}$ and $B_0=X\setminus Y$).

Now, for any $n\in\mathbb{N}$ and  $\beta\in\Gamma$, we apply  Lemma \ref{lemma:approximation} to $T_\beta-F$ on $B_{\beta}\cap Y$ to obtain a \Csmooth   map $\delta_{n,\beta}:X\longrightarrow \mathbb{R}$ satisfying conditions (C.1) , (C.2)  and
\begin{equation} \tag{C.3}
 \Lip(\delta_{n,\beta})\le C_1
\Lip(T_\beta-F)\le  C_1 (\Lip(f)+\Lip(F)).
\end{equation}
Besides, for all $n\in \mathbb N$ and  $\beta=0$, by applying property $(*)$, we select a $\mathcal{C}^1$-smooth
 function $F_0^n:X \rightarrow \mathbb R$,
 such that
\begin{equation*} 
|F_0^n(x)-F(x)|<\frac{\varepsilon}{2^{n+2}L_{n,0}} \text{ for every  } x\in  X \, \text{  and }
\Lip(F_0^n)\leq C_0 \Lip(F).
\end{equation*}
Thus, if we  define $\Delta_\beta^n:X \rightarrow \mathbb R$
\begin{equation}\label{deltanbetalip}
\Delta_\beta^n(x)=
\begin{cases} F_0^n(x) & \text{ if } \beta=0,\\
T_\beta(x)-\delta_{n,\beta}(x) & \text{ if } \beta\in\Gamma,
\end{cases}
\end{equation}
we obtain for every $\beta \in \Sigma$,
\begin{equation*}
|\Delta_\beta^n(x)-F(x)|<\frac{\varepsilon}{2^{n+2}L_{n,\beta}} \  \text{ whenever } x\in X\end{equation*}
 and $\Lip(\Delta_\beta^n)\leq \max\{(1+C_1) \Lip(f) +C_1\Lip(F),C_0 \Lip(F)\}\le R \Lip(F)$
 where $R:=1+2C_1$ is a constant depending only on $X$.
Similarly to the first case, the definition of $G$ is \begin{equation*}
G(x)=\sum_{(n,\beta)\in\mathbb{N} \times \Sigma} \psi_{n,\beta}(x)\Delta^n_\beta(x).
\end{equation*}
The proofs that $G$ is $\mathcal{C}^1$-smooth,
$|G(x)-F(x)|<\varepsilon $ for all $x\in X$  and $||g'(y)-f'(y)||_{Y^*}<\varepsilon$ for all $y \in Y$ (recall that  $g=G_{\mid_Y}$)  follow along the same lines.
In addition, let us check that $G$ is Lipschitz. From properties (D.1) to (D.6), in particular from the fact that  $\sum_{(n,\beta)\in F_x}\psi'_{n,\beta}(x)=0$, we deduce that
\begin{align}\label{cotaparaG'}
||G'(x)||_{X^*} &\le \sum_{(n,\beta)\in F_x}||\psi'_{n,\beta}(x)||_{X^*} |\Delta_\beta^n(x)-F(x)|+
\sum_{(n,\beta)\in F_x}\psi_{n,\beta}(x)||(\Delta_\beta^n)'(x)||_{X^*}\le \\  \notag
& \le  \sum_{\{n:(n,\beta(n))\in F_x\}}L_{n, \beta(n)}\, \frac{\varepsilon}{2^{n+2}L_{n, \beta(n)}}+
\sum_{\{n:(n,\beta(n))\in F_x\}}\psi_{n,\beta(n)}(x)R \Lip(F) \le \\  \notag
 &\le \frac{\varepsilon}{4}+R\Lip(F).
\end{align}
Since $\varepsilon < \Lip(F)$, then $\Lip(G)\leq C_2 \Lip(F)$ where $C_2:=R+\frac{1}{4}$ and this finishes the proof.
\end{proof}

The above theorem provides the tool to prove the  extension Theorem \ref{extension}, which states that every \Csmooth function (\Csmooth and Lipschitz function)  defined on a closed subspace has a \Csmooth extension (\Csmooth and Lipschitz extension, respectively) to $X$.

\noindent \emph{Proof of Theorem \ref{extension}. }
For every  bounded, Lipschitz function, $h:Y \to \Real$, we define
$\overline{h}(x):=  \min\{||h||_\infty, \inf_{y\in Y} \{h(y)+\Lip(h)||x-y||\}\}$ for any $x\in X$.
The function $\overline{h}$ is a Lipschitz extension of $h$ to $X$ with $\Lip(\overline{h})=\Lip(h)$ and
$||\overline{h}||_\infty=||h||_\infty=\sup\{|h(y)|:y\in Y\}$.

Let us assume that  the
function $f:Y\to \Real$ is  \Csmooth and consider $F:X\to\Real$  a continuous extension of $f$ to $X$ and $\varepsilon>0$.
We apply  Theorem \ref{laclave} to deduce the existence of a  \Csmooth function $G_1:X\to \Real$ such that if 
 $g_1:={G_1}_{\mid_Y}$, we have
\begin{enumerate}
\item[(i)]  $\displaystyle{|F(x)-G_1(x)|<\varepsilon/2}$ for $x\in X$, and
\item[(ii)] $\displaystyle{||f'(y)-g_1'(y)||_{Y^*}<\varepsilon/2C_2}$ for $y\in Y$.  Since $Y$ is convex, then we have $\Lip(f-g_1)\le \varepsilon/2C_2$.
\end{enumerate}
 The function $f-g_1$ is bounded by $\varepsilon/2$ and $\frac{\varepsilon}{2C_2}$-Lipschitz on $Y$.
 Thus, there exists a bounded, Lipschitz extension to $X$, $\overline{f-g_1}$, satisfying  $|\overline{(f-g_1)}(x)|\le\varepsilon/2$ on $X$ and $\Lip(\overline{f-g_1})\leq \varepsilon/2C_2$.
Following the construction given for the separable case \cite{Azafrykeener}, we apply  Theorem \ref{laclave} (Lipschitz case) to $\overline{f-g_1}$ to obtain  a \Csmooth function $G_2:X\to\Real$ such that if
 $g_2:={G_2}_{\mid_Y}$, we have
\begin{enumerate}
\item[(i)] $\displaystyle{|\overline{(f-g_1)}(x)-G_2(x)|<\varepsilon/2^2}$ for $x\in X$,
\item[(ii)]  $\displaystyle{||f'(y)-(g_1'(y)+g_2'(y))||_{Y^*}<\varepsilon/2^2C_2}$ for $y\in Y$, and
\item[(iii)]  $\displaystyle{\Lip(G_2)\le C_2 \Lip(f-g_1)\le \varepsilon/2}$.
\end{enumerate}
Thus, we find, by induction,  a sequence  $\{G_n\}_{n=1}^\infty$  of \Csmooth function such that for $n\ge 2$, the functions $G_n:X\to \Real$ and their restrictions $g_n:={G_n}_{\mid_Y}$ satisfy:
\begin{itemize}
\item[(i)]  $\displaystyle{|\overline{(f-\sum_{i=1}^{n-1} g_i)}(x)-G_n(x)|<\varepsilon/2^n}$ for $x\in X$,
\item[(ii)] $\displaystyle{||f'(y)-\sum_{i=1}^n g_i'(y)||_{Y^*} <\varepsilon/2^nC_2 }$ for $y\in Y$, and
\item[(iii)] $\displaystyle{\Lip(G_n)\leq C_2\Lip (f-\sum_{i=1}^{n-1} g_i)}\le \varepsilon/2^{n-1}$.
\end{itemize}
Let us define the function $H:X\to \Real$ as $H(x):=\sum_{n=1}^\infty G_n(x)$.  Since
$|G_n(x)|\le \varepsilon/2^{n-2}$ and $||G_n'(x)||_{X^*}\leq \Lip(G_n)\le \varepsilon/2^{n-1}$ for all $x\in X$ and $n\ge 2$,   the series
$\sum_{n=1}^\infty G_n$ and $\sum_{n=1}^\infty G_n'$ are absolutely and uniformly convergent on $X$. Hence, the function $H$ is \Csmooth on $X$.
It follows  from (i) that $|f(y)-\sum_{i=1}^n G_i(y)|<\varepsilon/2^n$ for every $y\in Y$ and $n\geq 1$.
Thus $H(y)=f(y)$  for all $y\in Y$.

Let us now consider $f:Y\to\Real$, \Csmooth and Lipschitz on $Y$. Let $F:X\to\Real$ be a Lipschitz extension of $f$ with $\Lip(F)=\Lip(f)$. We may assume $\Lip(f)>0$ (otherwise the result trivially holds) and take  $0<\varepsilon<\Lip(f)$.
Let us apply Theorem \ref{laclave} (Lipschitz case) to obtain a \Csmooth function $G_1:X\to \Real$ such that
if
 $g_1:={G_1}_{\mid_Y}$, we have
\begin{itemize}
\item[(i)]  $\displaystyle{|F(x)-G_1(x)|<\varepsilon/2}$ for $x\in X$,
\item[(ii)] $\displaystyle{||f'(y)-g_1'(y)||_{Y^*}<\varepsilon/2C_2}$ for $y\in Y$, and
\item[(iii)] $\displaystyle{\Lip(G_1)\le C_2\Lip(f)}$.
\end{itemize}
Let us define $G_n:X\to\Real$ for $n\ge 2$ as in the general case and $H(x):=\sum_{n=1}^\infty G_n(x)$. Then,  $H$ is $\mathcal{C}^1$-smooth, $H_{\mid_Y}=f$ and
\begin{equation*}
\Lip(H)\le \Lip(G_1)+\sum_{n=2}^\infty \Lip(G_n)\le C_2 \Lip(f) + \sum_{n=2}^\infty \frac{\varepsilon}{2^{n-1}}\leq (C_2+1)\Lip(f).
\end{equation*}
\qed


The proof of Corollary \ref{manifold} is similar to the separable case \cite{Azafrykeener}. (Recall that a paracompact $\mathcal{C}^1$ manifold $M$ modeled on a Banach space admits  $\mathcal{C}^1$-smooth partitions of unity whenever the Banach space where it is modeled  does.)


An analogous result to Theorem  \ref{laclave} can be stated for a smooth function defined on a closed, convex subset $Y$ of  $X$
with the required conditions  given  in Theorem \ref{extensionset}. Let us sketch  the required modifications
of the proof:  first, in the non-Lipschitz case,  we take  $H(f'(y_\beta))=f'(y_\beta)$ and we evaluate the norms of the functionals
in $X^*$ rather than in $Y^*$. Secondly, in the Lipschitz case, there is no loss of generality in assuming that $0\in Y$. Then, it can be easily checked that
$||f'(y)_{\mid_{Z}}||_{Z^*}\le \Lip(f)$ for all $y\in Y$, where we define  $Z:=\overline{\text{span}} (Y)$.
Next, we select  a continuous linear extension  of $f'(y_\beta)_{\mid_{Z}}$ to $X$ with the same norm and denote it by $H(f'(y_\beta))$  for every $\beta \in \Gamma$ (i.e. $||H(f'(y_\beta))||_{X^*}\le \Lip(f)$).  In this case,  assertion (ii) in Theorem  \ref{laclave} reads as
follows:  $||f'(y)-g'(y)||_{Z^*}<\varepsilon$ for every $y \in Y$. Finally, the proof of Theorem \ref{extensionset} is similar  to the proof of Theorem \ref{extension}.


\section{Appendix. ${C}^1$-smooth extensions of certain ${C}^1$-smooth functions defined on closed subsets of a Banach space.}

Let us finish this note by studying the case of the \Csmooth extension of a real-valued function $f$ defined on a (possibly non-convex) closed subset $Y$ of a Banach space $X$ with property $(*)$.
Unfortunately, we can find  examples of a real-valued function
$f$ defined on an open neighborhood $U$ of a   (non-convex) closed subset $Y$ of  $\mathbb R$
such that $f:U\rightarrow \mathbb R$   is differentiable at every point $y\in Y$, the mapping $Y \mapsto \mathbb R$, $y \mapsto f'(y)$ is continuous on $Y$  and $f_{\mid_Y}$ does not admit any \Csmooth extension to $\mathbb R$. Thus, in general, we cannot obtain a similar statement to Theorem \ref{extensionset} for a non-convex  closed subset $Y$ of $X$.
It is worth pointing out that, by an application of the Mean Value Theorem, if  $f:Y\rightarrow \mathbb R$ admits a \Csmooth  extension $H:X\rightarrow \mathbb R$, then, setting  $D(y):=H'(y)\in X^*$  for $y\in Y$, we have the following mean value condition:
\begin{align}
\tag{E}
\text{ for all } y\in Y, \text{ for all } \varepsilon >0, \text{ there exists } r>0 \text{ such that} \qquad \quad \\
|f(z)-f(w)-D(y)(z-w)|\le \varepsilon||z-w||, \quad \text{  for all } z,w\in Y\cap B(y,r). \notag \end{align}

\begin{defn} If $Y$ is a closed subset of a Banach space $X$, let  $Z:=\overline{\text{\emph{span}}} \{y-y': y, y'\in Y\}$. Let  $f:Y\rightarrow \mathbb R$ and let  $D:Y\rightarrow Z^*$ a continuous mapping.
We  will then  say
that $f:Y\rightarrow \mathbb R$ satisfies condition $(E)$ for $D$ if the statement $(E)$ is true.
\end{defn}
Thus, condition (E) is necessary to  extend  $f:Y\rightarrow \mathbb R$ to a
 \Csmooth function $H:X\rightarrow \mathbb R$.
Moreover, the next result states that condition (E) is also sufficient  in spaces with property $(*)$.
\begin{thm}\label{extensionencerrados} Let $X$ be a Banach space with property $(*)$, $Y\subset X$ a closed subset and $f:Y\rightarrow \mathbb R$ a function on $Y$. Then,
 $f$ satisfies condition
$(E)$ if and only if there is a \Csmooth
extension $H$ of $f$ to $X$.

Furthermore, assume that the given  function $f$ is Lipschitz on $Y$. Then,
 $f$ satisfies condition $(E)$  for some continuous function $D:Y\to Z^*$ with
  $\sup\{||D(y)||_{Z^*}:y\in Y\}<\infty$ if and only if there is a \Csmooth and Lipschitz extension
$H$ of $f$ to $X$. In this  case, if $f$ satisfies condition $(E)$  for $D$ with  $M:=\sup\{||D(y)||_{Z^*}:y\in Y\}<\infty$,
then we can obtain $H$ with $\Lip(H)\le (1+C_1)(M+\Lip(f))$, where $C_1$ is the constant defined in Lemma \ref{lemma:approximation}. (Recall that $C_1$ depends only on $X$.)
\end{thm}

The proof of Theorem \ref{extensionencerrados} follows the same lines as Theorem
\ref{extension}.  It can be deduced from Lemmas \ref{partition:unity}, \ref{lemma:approximation} and the following  statement, which is  similar to Theorem \ref{laclave}.
\begin{thm}\label{llavedos}
 Let $X$ be a Banach space with property $(*)$, $Y\subset X$ a closed subset
 and $f:Y\rightarrow \mathbb R$ a  function satisfying condition $(E)$ for some continuous
 function $D:Y\to Z^*$.
 Let us consider $F$ a continuous
extension of $f$ to $X$. Then, for every $\varepsilon>0$ there exists a \Csmooth function $G:X\longrightarrow\mathbb{R}$ such that if $g:=G_{\mid_Y}$ then
\begin{enumerate}
\item[{(i)}] $|F(x)-G(x)|< \varepsilon$ on $X$, and
\item[{(ii)}] $||D(y)-G'(y)||_{Z^*} < \varepsilon$ for all $y\in Y$ and $\Lip(f-g) < \varepsilon$.
\item[{(iii)}] Furthermore, assume that $f$ is Lipschitz on $Y$, $F$ is a Lipschitz extension of $f$ to $X$ and $M=\sup\{||D(y)||_{Z^*}:y\in Y\}<\infty$. Then the function $G$ can be chosen to be Lipschitz on $X$ and $\Lip(G)\leq \frac{\varepsilon}{4}+(1+C_1)M+C_1\Lip(F)$.
\end{enumerate}

\end{thm}

Let us outline the required modifications of the proof of Theorem \ref{llavedos}.
The covering  $\{B(y_\gamma,r_\gamma):=B_\gamma\}_{\gamma\in\Gamma}$ of $Y$ by open balls of $X$  with centers $y_\gamma \in Y$ is selected
in such a way that
\begin{equation*}
||D(y)-D(y_\gamma)||_{Z^*}\le\frac{\varepsilon}{8C_0}
\quad
\text{  and  }
\quad
|f(z)-f(w)-D(y_\gamma)(z-w)| \le \frac{\varepsilon}{8C_0}  ||z-w|| ,\end{equation*}
 for every $y,z,w\in B_\gamma\cap Y$ (this can be done because $f$ satisfies condition  $(E)$ for $D$).
In this case, $T_\gamma$ is defined as
$T_\gamma(x)=f(y_\gamma)+\widehat{D}(y_\gamma)(x-y_\gamma)$, for $x\in X$, where $\widehat{D}(y_\gamma)$
is an extension of $D(y_\gamma)$ to $X$ with the same norm.
The functions $T_\gamma$ are $\mathcal{C}^\infty$-smooth on $X$,
     $T'_\gamma(x)=\widehat{D}(y_\gamma)$ for all $x\in X$ and satisfy that for all
     $z,w \in B_\gamma \cap Y$,
\begin{equation*}
|(T_\gamma-F)(z)-(T_\gamma-F)(w)|=|f(w)-f(z)-D(y_\gamma)(w-z)| \le\frac{\varepsilon}{8C_0}  ||z-w||.
\end{equation*}
Thus, $\Lip((T_\gamma-F){\mid_{B_{\gamma}\cap Y}})\le\frac{\varepsilon}{8C_0}.$
 We can apply Lemma \ref{lemma:approximation} to $T_\gamma-F$ on $B_{\gamma}\cap Y$ to obtain a \Csmooth   map $\delta_{n,\gamma}:X\longrightarrow \mathbb{R}$ satisfying (C.1), (C.2) and
 \begin{equation}\tag{C.2'}\Lip((\delta_{n,\gamma})_{\mid_{B_\gamma\cap Y}})\le \frac{\varepsilon}{8}.\end{equation}
From the above, we deduce that,  for all $y\in B_\gamma \cap Y$,
\begin{equation*}
||T_\gamma'(y)-D(y)-\delta_{n,\gamma}'||_{Z^*}\le \frac{\varepsilon}{4}
\quad \text{ and } \quad
\Lip((T_\gamma-F-\delta_{n,\gamma})_{\mid_{B_\gamma\cap Y}}) \le \frac{\varepsilon}{4}.
 \end{equation*}
The inequality $||D(y)-G'(y)||_{Z^*}< \varepsilon$ follows as in \eqref{f'-g'} (in the proof of Theorem \ref{laclave}).
Let us prove that $\Lip(g-f)<\varepsilon$, where $g=G_{\mid_Y}$. In order to simplify the notation
let us write $S^n_\beta(y):=\Delta^n_\beta(y)-f(y)$ for $y\in Y$ (where $\Delta^n_\beta$ is defined as in \eqref{deltanbeta}), 
and  $R(y,z):=\sum_{(n,\beta)\in F_y} \psi_{n,\beta}(z)S^n_\beta(y)$ for $y,z\in Y$. Notice that $\psi_{(n,\beta)}(z)=0$ whenever $(n,\beta)\not\in F_z$ and  thus, $R(y,z)=\sum_{(n,\beta)\in F_y\cap F_z} \psi_{n,\beta}(z)S^n_\beta(y)$ for $y,z\in Y$.
In addition, let us write $M(z,y):=\sum_{(n,\beta)\in F_z\setminus F_y} \psi_{n,\beta}(y)S^n_\beta(z)=0$ for $y,z\in Y$.
Now, from the above and properties (D.1) to (D.6), we obtain
\begin{align*}
 &|(g(y)-f(y))-(g(z)-f(z))|=& \\
 &=| \sum_{(n,\beta)\in F_y} \psi_{n,\beta}(y)S^n_\beta(y)-
 \sum_{(n,\beta)\in F_z}\psi_{n,\beta}(z) S^n_\beta(z)
-R(y,z) +{R}(y,z) +M(z,y)|=\\ &=
 | (\sum_{(n,\beta)\in F_y} \psi_{n,\beta}(y)S^n_\beta(y)-R(y,z) ) + ({R}(y,z)-
 \sum_{(n,\beta)\in F_z\cap F_y}\psi_{n,\beta}(z) S^n_\beta(z)
 ) \\
& \hspace{6.5cm}+ (M(z,y)-\sum_{(n,\beta)\in F_z\setminus F_y}\psi_{n,\beta}(z) S^n_\beta(z))|\le\\
 & \le \sum_{(n,\beta)\in F_y}|\psi_{n,\beta}(y)-\psi_{n,\beta}(z)||S^n_\beta(y)|
 +\sum_{(n,\beta)\in F_z\cap F_y}\psi_{n,\beta}(z)|S^n_\beta(y)-S^n_\beta(z)| + \\
 &\quad +\sum_{(n,\beta)\in F_z\setminus F_y}|\psi_{n,\beta}(y)-\psi_{n,\beta}(z)||S^n_\beta(z)|\le \sum_{(n,\beta)\in F_y}L_{n,\beta}||y-z|| \,\frac{\varepsilon}{2^{n+2}L_{n,\beta}}+
\\
 &    \quad +\sum_{(n,\beta)\in F_z\cap F_y}\psi_{n,\beta}(z) \frac{\varepsilon}{4} ||y-z||   +  \sum_{(n,\beta)\in F_z\setminus F_y}L_{n,\beta}||y-z|| \,\frac{\varepsilon}{2^{n+2}L_{n,\beta}}  <\varepsilon||y-z||.&
\end{align*}
In the Lipschitz case,  we additionally obtain that
$T_\beta-F$ is Lipschitz on $X$ and
$\Lip(T_\beta-F)\le \Lip(T_\beta)+\Lip(F) \le M+ \Lip(F)$ for every $\beta \in \Gamma$.
Thus, $\delta_{n,\beta}:X\longrightarrow \mathbb{R}$ satisfies conditions (C.1), (C.2), (C.2') and
$\Lip(\delta_{n,\beta})\le C_1\Lip(T_\beta-F)\le   C_1(M+\Lip(F))$.
Also, if $\Delta_\beta^n$ is defined as in \eqref{deltanbetalip}, then
$\Lip(\Delta_\beta^n)\leq \max\{(1+C_1) M +C_1\Lip(F),C_0 \Lip(F)\}=(1+C_1) M +C_1\Lip(F)$.
Similarly to  \eqref{cotaparaG'}, we obtain the following upper bound,
\begin{equation*}
||G'(x)||_{X^*}\le  \frac{\varepsilon}{4}+(1+C_1) M +C_1 \Lip(F).
\end{equation*}
(Recall, that here we do not assume $\varepsilon < \Lip(F)$.) This finishes the required changes for the
proof of Theorem \ref{llavedos}.

\medskip

Let us give also the necessary modifications for the proof of Theorem \ref{extensionencerrados}.
In this case, we apply  Theorem \ref{llavedos} to deduce the existence of a  \Csmooth function $G_1:X\to \Real$ such that if  $g_1:= {G_1}_{\mid_Y}$, we have
\begin{enumerate}
\item[(i)]  $|F(x)-G_1(x)| < \frac{\varepsilon}{2(1+C_1)}$ for $x\in X$,
\item[(ii)] $||D(y)-G_1'(y)||_{Z^*} < \frac{\varepsilon}{2(1+C_1)}$ for all $y\in Y$ and $\Lip(f-g_1) < \frac{\varepsilon}{2(1+C_1)}$.
\end{enumerate}
The function $f-g_1$  satisfies condition $(E)$ for
  $D-G_1'$. Thus, we apply again Theorem \ref{llavedos} (Lipschitz case) to $\overline{f-g_1}$ to obtain
a \Csmooth function $G_2:X\to\Real$ such that  if  $g_2:= {G_2}_{\mid_Y}$, we have
\begin{enumerate}
\item[(i)] $|\overline{(f-g_1)}(x)-G_2(x)| < \frac{\varepsilon}{2^2(1+C_1)}$ for $x\in X$,
\item[(ii)] $||D(y)-(G_1'(y)+G_2'(y))||_{Z^*} < \frac{\varepsilon}{2^2(1+C_1)}$ for all $y\in Y$,
   $\Lip(f-(g_1+g_2)) < \frac{\varepsilon}{2^2(1+C_1)}$, and
\item[(iii)]  $\Lip(G_2)\le\frac{\varepsilon}{2^4}+\varepsilon$.
\end{enumerate}
We find, by induction,  a sequence  $\{G_n\}_{n=1}^\infty$  of \Csmooth function such that for $n\ge 2$, the functions $G_n:X\to \Real$ and their restrictions $g_n:={G_n}_{\mid_Y}$ satisfy:
\begin{itemize}
\item[(i)]  $|\overline{(f-\sum_{i=1}^{n-1} g_i)}(x)-G_n(x)| < \frac{\varepsilon}{2^n(1+C_1)}$ for $x\in X$,
\item[(ii)] $||D(y)-\sum_{i=1}^n G_i'(y)||_{Z^*} < \frac{\varepsilon}{2^n(1+C_1)}$ for $y\in Y$, $\Lip (f-\sum_{i=1}^{n} g_i) < \frac{\varepsilon}{2^n(1+C_1)}$, and
\item[(iii)] $\Lip(G_n)\le \frac{\varepsilon}{2^{n+2}}+\frac{\varepsilon}{2^{n-2}}$.
\end{itemize}
In the Lipschitz case and for $\varepsilon<\frac{4}{9}\Lip(f)$, we obtain the upper bound,
$\Lip(H)\le (1+C_1)(M+\Lip(f))$.

\qed

It is worth pointing out the following corollary, where a characterization
of property $(*)$ is given.
\begin{cor}
Let X be a Banach space. The following statements are equivalent:
\begin{itemize}
\item[(i)] $X$ satisfies property $(*)$.
\item[(ii)] Assume that $f:Y\to\Real$ is a Lipschitz function satisfying condition $(E)$ for some continuous function $D:Y\to Z^*$ with
$M:=\sup\{||D(y)||_{Z^*}:y\in Y\}<\infty$.
Then, there is a \Csmooth and Lipschitz extension of $f$ to $X$,
$H:X\to\Real$,  with $\Lip(H)\le C_3(M+\Lip(f))$ (where $C_3$ is a constant depending
only on the space $X$).
\end{itemize}
\end{cor}

\begin{proof}
We only need to check $(ii) \Rightarrow (i)$.   By  \cite[Proposition 1]{Hajek} it is enough to  prove that  there is a number $K\geq 1$ such that for every subset $A\subset X$ there is a \Csmooth and $K$-Lipschitz function $h_A:X\to [0,1]$ such that $h_A(x)=0$ for $x\in A$ and $h_A(x)=1$ for every $x\in X$ such that $\dist(x,A)\geq 1$.
For every subset $A$  of $X$, we consider the closed subsets $B=\{x\in X: \dist(x,A)\geq 1\}$ and $Y=\overline{A}\cup B$ of $X$. Let us define  the function $f:Y\to \Real$ as
\begin{equation*}
f(x):=
\begin{cases}
0 & \text{if $x\in \overline{A}$,}\\
1 & \text{if $x\in B$.}
\end{cases}
\end{equation*}
It is clear that  $f$ is $1$-Lipschitz and  satisfies condition $(E)$ for the continuous function
$D:Y\rightarrow Z^*$, $D(y)=0$ for every $y\in Y$ (thus,
$\sup\{||D(y)||_{Z^*}: y\in Y\}=0$).
By assumption, we can find a  \Csmooth and $C_3$-Lipschitz extension of $f$ to $X$,
$H:X\to \Real$ such that $H(x)=0$ for all $x\in A$ and $H(x)=1$ for all $x\in B$.
We take a $2$-Lipschitz and $C^\infty$-smooth function  $\theta:\Real \to [0,1]$ such that $\theta(t)=0$ whenever $t\leq 0$, and $\theta(t)=1$ whenever $t\geq1$. Let us define $h_A(x)=\theta(H(x))$. The \Csmooth function $h_A:X\rightarrow [0,1]$ is $2C_3$-Lipschitz,  $h_A(x)=0$ for $x\in A$ and $h_A(x)=1$ for $x\in B$.
\end{proof}

\medskip

Notice that  the following assertion can be proved along the same lines:
if every  real-valued  function  defined on a closed subset of a Banach space $X$ and satisfying condition (E) has a  \Csmooth extension to $X$, then every real-valued, continuous function on $X$ can be uniformly approximated by \Csmooth functions. In fact, both properties are equivalent in the separable case to  properties (i) and (ii) of the above corollary.

\medskip


Finally, let us mention, that after submitting the paper to the journal, we were informed of  the final version of \cite{Hajek}, where a similar statement to Lemma \ref{partition:unity} is established. In any case, we have decided to keep this lemma in order to give a more self-contained proof of the main result.




\end{document}